\documentclass[12pt]{amsart} %amsfonts,
\usepackage{amsmath,amssymb,amscd,amsthm}
\usepackage{latexsym}
\usepackage{graphicx}
\usepackage[english]{babel}
\usepackage[latin1]{inputenc}       % allow Latin1 characters
\usepackage{textcomp}
\usepackage{enumitem}
\usepackage{times}
\usepackage{pgf,pgfarrows,pgfnodes,pgfautomata,pgfheaps}
\usepackage{colortbl}
\usepackage{color}
\usepackage[a4paper,twoside,left=3cm,right=2.8cm,top=3.1cm,bottom=2.3cm]{geometry}

\newtheorem{theorem}{Theorem}[section]
\newtheorem{corollary}[theorem]{Corollary}
\newtheorem{proposition}[theorem]{Proposition}
\newtheorem{lemma}[theorem]{Lemma}

\newtheorem{definition}[theorem]{Definition}

\theoremstyle{definition}
\newtheorem{remark}[theorem]{Remark}

\def\irr#1{{\rm Irr}(#1)}
\def\irrr#1#2 {\irr {#1 \mid #2}}
\newcommand{\R}{\mathbb R}

\newcommand{\s}{\mathbb S}
\newcommand{\sfe}{{{\mathbb S}^{n-1}}}

\begin{document}
%\color{blue}

\title[Extending Minkowski's theorem; the Shephard problem for measures]{An extension of Minkowski's theorem and its applications to questions about projections for measures.}
\author[Galyna V. Livshyts]{Galyna V. Livshyts}
\address{Georgia Institute of Technology}
\email{glivshyts6@math.gatech.edu}
%\thanks{The author would like to thank Alex Koldobsky, Artem Zvavitch and Liran Rotem for very useful fruitful discussions.}
\subjclass[2010]{Primary: 52} 
\date{\today}
\begin{abstract} 

Minkowski's Theorem asserts that every centered measure on the sphere which is not concentrated on a great subsphere is the surface area measure of some convex body, and the surface area measure determines a convex body uniquely up to a shift. In this manuscript we prove an extension of Minkowski's theorem. Consider a measure $\mu$ on $\R^n$ with positive degree of concavity and positive degree of homogeneity. We show that a surface area measure of a convex set $K$, weighted with respect to $\mu$, determines a convex body uniquely up to $\mu$-measure zero. We also establish an existence result under natural conditions including symmetry.

We apply this result to extend the solution to classical Shephard's problem. To do this, we introduce a new notion which relates projections of convex bodies to a given measure $\mu$, and is a generalization of the Lebesgue volume of a projection. 
\end{abstract}
\maketitle

\section{Introduction}

We shall work in an $n-$dimensional vector space $\R^n$ with standard orthonormal basis $e_1,...,e_n$ and  a scalar product $\langle \cdot,\cdot\rangle$. The standard Euclidean length is denoted by $|\cdot|$. 

A set $K$ in $\R^n$ is said to be convex if together with every pair of points it contains the interval connecting them. Compact convex sets with non-empty interior are called convex bodies. 

The standard Lebesgue measure of a set $A$ in $\R^n$ shall be denoted by $|A|$ or, sometimes, $|A|_n$. When the standard Lebesgue measure on a subspace of dimension $k$ is considered, it shall be denoted by $|\cdot|_k$. We shall denote the unit ball centered at the origin in $\R^n$ by $B_2^n$, and the unit sphere by $\sfe.$ 

Given a convex body $K$ in $\R^n$, its Gauss map $\nu_K:\partial K\rightarrow \sfe$ is the map that corresponds to every $y\in \partial K$ the set of normal vectors at $y$ with respect to $K.$ The surface area measure of $K$ is the measure on the unit sphere defined as the push forward to the sphere of the $(n-1)-$dimensional Hausdorff measure on $\partial K$ via the map $\nu_K$. It is denoted by $\sigma_K.$

Minkowski's existence theorem guarantees that every barycentered measure on $\sfe$ which is not supported on any great subsphere is a surface area measure for some convex body; moreover, a convex body is determined by its surface area measure uniquely up to a shift. 

For $p\in\R$, the $L_p$ surface area measure of a convex body with the support function $h_K$ is the measure on the sphere given by $d\sigma_{p,K}(u)=h_K^{1-p}(u) d\sigma_{K}(u)$. It was introduced by Lutwak. The normalized $L_p$ surface area is given by $d\bar{\sigma}_{p,K}(u)=\frac{1}{|K|}d\sigma_{p,K}(u)$. An extension of Minkowski's Theorem, called \emph{$L_p-$Minkowski problem} is open in general. It asks which conditions should be required in order for a measure on the sphere to be an $L_p-$surface area measure, as well as whether $L_p-$surface area measure determines a convex body uniquely. Lutwak, Yang, Zhang have solved the normalized $L_p$-Minkowski problem with even data for the case $p\leq 0$, and showed the uniqueness of the solution when $p<0$. B\"or\"oczky, Lutwak, Yang, Zhang \cite{BLYZ}, \cite{BLYZ-1}, \cite{BLYZ-2} have studied the case $p=0$ and have, in particular, obtained the uniqueness in the case of symmetric convex bodies on the plane. Stancu \cite{St}, \cite{St1} has treated this problem for polytopes on the plane. Huang, Liu, Xu \cite{HLX} have established uniqueness in $\R^3$ in the case when the $L_p$ surface area is constant. The $L_p$-Minkowski problem is one of the main questions in the rapidly developing Brunn-Minkowski-Firey theory (see more in Ludwig \cite{Lud}, Lutwak \cite{Lut1}, \cite{Lut2}, Lutwak, Yang, Zhang \cite{LYZ-1}, \cite{LYZ-2}, \cite{LYZ-3}, Lutwak, Oliker \cite{LO},  Meyer, Werner \cite{MW}, Ryabogin, Zvavitch \cite{RZ}, Zhu \cite{zhu}, \cite{zhu1}, and the references therein).

In this manuscript, we prove an analogue of Minkowski's theorem in a different setting. Let $\mu$ be an absolutely continuous measure on $\R^n$. We study the surface area measure of convex bodies with respect to $\mu$. 

\begin{definition}
Let $K$ be a convex body and $\nu_K$ be its Gauss map. Let $\mu$ be a measure on $\R^n$ with density $g(x)$ continuous on its support. Define $\sigma_{\mu,K}$ on $\sfe$, a surface area measure of $K$ with respect to $\mu$, as follows: for every Borel set $\Omega\subset\sfe$, let
$$\sigma_{\mu,K}(\Omega)=\int_{\nu_K^{-1}(\Omega)} g(x) dH_{n-1}(x),$$
where $H_{n-1}$ stands for the $(n-1)-$dimensional Hausdorff measure on $\partial K$, and $\nu_K^{-1}(\Omega)$ stands for the full pre-image of $\Omega$ under $\nu_K.$
\end{definition}

When $\mu$ is the standard Lebesgue measure, the measure $\sigma_{\mu, K}$ coincides with $\sigma_K,$ the classical surface area measure.

Let $p\in (0,+\infty)$. We say that a function $f:\R^n\rightarrow [0,\infty]$ is $p-$concave if $f^p(x)$ is a concave function on its support. That is, for every $x,y\in supp(f)$ and for every $\lambda\in [0,1]$ we have 
$$f^p(\lambda x+(1-\lambda) y)\geq \lambda f^p(x)+(1-\lambda) f^p(y).$$

Let $r\in (-\infty,+\infty)$. We say that a function $f:\R^n\rightarrow [0,\infty]$ is $r-$ homogenous if for every $a>0$ and for every $x\in\R^n$ we have $f(ax)=a^{r} f(x)$. 
 
We shall consider the class of measures on $\R^n$ with densities that have a positive degree of homogeneity and a positive degree of concavity. In fact, all such densities are $p-$concave and $\frac{1}{p}$-homogenous for the same $p\geq 0$ (see the Proposition \ref{p-concave-proposition} from the Appendix). This class of measures was considered by E. Milman and L. Rotem \cite{MilRot}, where they studied their isoperimetric properties. We remark that such measures are necessarily supported on convex cones. An example of a density function with said properties is $f(x)=1_{\{\langle x,\theta\rangle>0\}}|\langle x, \theta\rangle|^{\frac{1}{p}}$, where $\theta$ is a vector.

We prove an extension of Minkowski's existence theorem to the class of surface area measures with respect to measures with positive degree of concavity and positive degree of homogeneity. 

\begin{theorem}\label{mink-existence-intro}
Let $\mu$ on $\R^n$ be a measure and $g(x)$ be its even $r-homogenous$ density for some $r\geq 0$, and the restriction of $g$ to some half space is $p-concave$ for a $p\geq 0$. Let $\varphi(u)$ be an arbitrary even measure on $\sfe$, not supported on any great subsphere, such that $supp(\varphi)\subset int(supp(g))\cap \sfe$. Then there exists a symmetric convex body $K$ in $\R^n$ such that 
$$d\sigma_{K,\mu}(u)=d\varphi(u).$$
Moreover, such convex body is determined uniquely up to a set of $\mu-$measure zero.
\end{theorem}

In Theorem \ref{mink-existence-intro}, and throughout the paper, uniqueness up to $\mu-$measure zero means that for every pair of $K$ and $L$, symmetric convex bodies with $\sigma_{K,\mu}=\sigma_{L,\mu}$, the measure of their symmetric difference $\mu(K\Delta L)=0.$

We apply Theorem \ref{mink-existence-intro} to extend the study of volume comparison and unique determination of convex bodies related to projections.

Given a unit vector $u\in\sfe$, we consider an $(n-1)-$dimensional hyperplane orthogonal to it:
$$u^{\perp}=\{x\in\R^n\,:\, \langle x,u\rangle=0\}.$$

An orthogonal projection of a convex body $K$ to a subspace $u^{\perp}$ shall be denoted by $K|u^{\perp}$; that is,
$$K|u^{\perp}=\{x\in u^{\perp}\,:\, \exists t\in \R\, s.t.\, x+tu\in K\}.$$

Let K be an origin symmetric convex body in $\R^n$ with curvature function $f_K$. The projection body $\Pi K$ of $K$ is defined as the origin symmetric convex body in $\R^n$ whose support function in every direction is equal to the volume of the hyperplane projection of K in this direction.

The Shephard problem (see Shephard \cite{shep}) is the following question: \emph{given symmetric convex bodies $K$ and $L$ such that for every $u\in \sfe$ 
$$|K|u^{\perp}|_{n-1}\leq |L| u^{\perp}|_{n-1},$$
does it follow that $|K|_n\leq |L|_n$?}
The problem was solved independently by Petty \cite{P} and Schneider \cite{Sc1}. They showed that the answer is affirmative if $n\leq 2$ and negative if $n\geq 3$. More precisely, the answer to Shephard's problem is affirmative if and only if $L$ is a projection body. As for general symmetric convex bodies, Ball \cite{B} proved that if the volumes of projections of $K$ are less than or equal to the volumes of projections of $L$ in every direction, then $|K|\leq \sqrt{n} |L|$, for every dimension $n$. Goodey and Zhang \cite{GZ} obtained a generalization of the Shephard problem for lower dimensional projections. A Fourier analytic approach to Shephard's problem was presented by Koldobsky, Ryabogin and Zvavitch \cite{KRZ}. Ryabogin and Zvavitch \cite{RZ} solved the generalization of Shephard's problem for Firey projections.

The Busemann-Petty problem is in a sense dual to the Shephard problem. It asks whether symmetric convex bodies with larger central hyperplane sections necessarily have greater volume. The Busemann-Petty problem has been solved affirmatively for $n\leq 4$ and negatively for $n\geq 5$ (see Gardner, Koldobsky, Schlumprecht \cite{GKS} and Zhang \cite{zhang}). The answer to Busemann-Petty problem is affirmative if and only if the body with larger sections is an intersection body (see Lutwak \cite{Lu} for the definition and properties of intersection bodies, and Koldobsky \cite{K6} for Fourier analytic approach to intersection bodies). Zvavitch solved an isomorphic version of Busemann-Petty problem for Gaussian measures \cite{zvav1}, and completely generalized the solution of Busemann-Petty problem to arbitrary measures with positive density \cite{zvav}. Koldobsky \cite{K8}, and further Koldobsky and Zvavitch \cite{KZ} obtained estimates for the isomorphic version of Busemann-Petty problem for arbitrary measures; a discrete analog of those estimates was very recently obtained by Alexander, Zvavitch, Henk \cite{AHZ}.

We refer the reader to the books by Koldobsky \cite{Kold} and Koldobsky, Yaskin \cite{KY} for a deep, yet accessible study of the Fourier-analytic approach to the Busemann-Petty and Shephard problems, as well as a general introduction to Fourier analysis in Convex geometry.

Aleksandrov in \cite{A2} proved that any symmetric convex body in $\R^n$ is determined uniquely by the $(n-1)-$dimensional volumes of its projections. See Zhang \cite{zhang} for the discrete version of that statement under natural assumptions. In Section 5 we generalize Aleksandrov's theorem to measures with positive degree of concavity and positive degree of homogeneity. 

First, we find a natural analogue of the Lebesgue measure of projection of a convex body to other measures. 
\begin{definition}\label{proj_gen}
Let $\mu$ be a measure on $\R^n$ with density $g$ continuous on its support, and let $K$ be a convex body. Consider a unit vector $\theta\in \sfe$. Define the following function on the cylinder $\sfe\times [0,1]$:
\begin{equation}\label{mu-proj-t}
p_{\mu, K}(\theta,t):=\frac{n}{2}\int_{\sfe} |\langle \theta, u\rangle| d\sigma_{\mu, tK}(u).
\end{equation}
We also consider $\mu-projection$ function on the unit sphere:
\begin{equation}\label{mu-proj}
P_{\mu, K}(\theta):=\int_0^1 p_{\mu, K}(\theta,t) dt.
\end{equation}
\end{definition}

In the particular case of Lebesgue measure $\lambda$ we have
$$P_{\lambda, K}(\theta)=|K|\theta^{\perp}|_{n-1}.$$

The Definition \ref{proj_gen} is natural since it is a generalization of Cauchy's projection formula (see below (\ref{Cauchy})). For even $g$, the notion of $p_{\mu,K}(\theta,t)$ can be understood geometrically as the projected weight of the boundary of $tK$, $t\in[0,1]$. More specifically, we define a measure $\mu_{tK}$ on $\theta^{\perp}$ to be the marginal measure of $1_{\partial (tK)}(x)g(x)dx$. In other words, for a measurable set $\Omega\subset \theta^{\perp}$, let $\mu_{tK}(\Omega)=\int_{\Omega} g(\pi_{tK}^{-1}(w)) dw$, where $\pi_{tK}^{-1}(w)$ is the full pre-image of $w$ under the projection of $tK$ onto $\theta^{\perp}$. Then
$$p_{\mu,K}(\theta, t)=\mu_{tK}(tK|\theta^{\perp})=\mu_{tK}(K|\theta^{\perp}),$$
where the last equality holds since $tK\subset K$. Hence, 
$$P_{\mu,K}(\theta)=\int_0^1\mu_{tK}(K|\theta^{\perp})dt.$$

We prove the following result.

\begin{theorem}\label{main}
Fix $n\geq 1$; let $\mu$ on $\R^n$ be a measure and $g(x)$ be its even $r-homogenous$ density for some $r\geq 0$, and the restriction of $g$ to some half space is $p-concave$ for a $p\geq 0$.

Let $K$ and $L$ be symmetric convex bodies, and let $L$ additionally be a projection body. Assume that for every $\theta\in \sfe$ we have 
$$P_{\mu, K}(\theta)\leq P_{\mu, L}(\theta).$$
Then $\mu(K)\leq \mu(L)$.
\end{theorem}
To compliment Theorem \ref{main} we prove the following.
\begin{theorem}\label{main-2}
Fix $n\geq 1$; let $\mu$ on $\R^n$ be a measure and $g(x)$ be its even $r-homogenous$ density for some $r\geq 0$, and the restriction of $g$ to some half space is $p-concave$ for a $p\geq 0$. Assume further that the closure of the support of $\mu$ is the whole space.

Let $L$ be a symmetric convex body which is not a projection body. Then there exists a symmetric convex body $K$ such that for every $\theta\in \sfe$ we have 
$$P_{\mu, K}(\theta)\leq P_{\mu, L}(\theta),$$
but $\mu(K)> \mu(L)$.
\end{theorem}

We remark that in the case of Lebesgue measure Theorems \ref{main} and \ref{main-2} are generalizations of the well-known solution to the classical Shephard problem (see Koldobsky \cite{Kold}, Chapter 8). %We compliment our study by proving stability and separation results for Theorem \ref{main}.

This paper is organized as follows. In Section 2 we present the preliminaries on the subject. In Section 3 we introduce and study the notion of mixed measure and prove an analogue of Minkowski's first inequality for measures. In Section 4 we prove Theorem \ref{mink-existence-intro}. In Section 5 we prove two types of uniqueness results: one is the extension of Aleksandrov's theorem, and the other is related to the uniqueness of the solution of certain PDE in the class of support functions. In Section 6 we prove Theorems \ref{main} and \ref{main-2}. In Section 7 we discuss stability and separation results for Theorem \ref{main}, and their corollaries.

\textbf{Acknowledgement.} The author would like to thank Alex Koldobsky, Artem Zvavitch, Liran Rotem and Ben Jaye for very fruitful discussions and encouragement.

\section{Preliminaries}
\subsection{Brunn-Minkowski theory.}
Below we present classical concepts and results of Convex geometry and Brunn-Minkowski theory. We refer the reader to books by Ball \cite{Ball-book}, Milman, Schechtman \cite{MS}, Schneider \cite{book4} for a detailed introduction to the subject.

Standard Minkowski's addition for sets $A, B\subset \R^n$ is defined as
$$A+B:=\{a+b\,:\,a\in A, b\in B\}.$$ 
Scalar multiplication for $\alpha\in \R$ and a set $A\subset \R^n$ is defined as
$$\alpha A:=\{\alpha a\,:\, a\in A\}.$$
For Borel sets $A$, $B$ in $\R^n$ and for arbitrary $\lambda\in [0,1]$, Brunn-Minkowski inequality states that
$$|\lambda A+(1-\lambda) B|^{\frac{1}{n}}\geq \lambda |A|^{\frac{1}{n}}+(1-\lambda) |B|^{\frac{1}{n}}.$$
See Gardner \cite{gardner} for an exhaustive survey on the subject. We remark that for convex bodies the equality in the Brunn-Minkowski inequality is attained if and only if the sets $A$ and $B$ are closed, convex dilates of each other.

First mixed volume of convex bodies $K$ and $L$ in $\R^n$ is defined as follows:
$$V_1(K,L):=\frac{1}{n}\liminf_{\epsilon\rightarrow 0} \frac{|K+\epsilon L|-|K|}{\epsilon}.$$
Note that for any convex body $K$ one has
\begin{equation}\label{1mink-volume}
V_1(K,K)=|K|.
\end{equation}
Brunn-Minkowski inequality implies Minkowski's first inequality:
\begin{equation}\label{1mink}
V_1(K,L)\geq |K|^{\frac{n-1}{n}} |L|^{\frac{1}{n}}.
\end{equation}
There is equality in Minkowski's first inequality if and only if $K$ and $L$ are closed convex dilates of each other (see Schneider \cite{book4} for more details). 

A particular case of mixed volume, is the surface area of a convex set $K$ in $\R^n$:
$$|\partial K|^+:=nV_1(K,B_2^n)=\liminf_{\epsilon\rightarrow 0} \frac{|K+\epsilon B_2^n|-|K|}{\epsilon}.$$
Therefore, (\ref{1mink}) implies classical isoperimetric inequality:
$$\frac{|\partial K|^+}{|K|^{\frac{n-1}{n}}}\geq \frac{|\partial B_2^n|^+}{|B_2^n|^{\frac{n-1}{n}}}.$$

Next, we shall discuss Brunn-Minkowski inequality for $p-$concave measures (see Gardner \cite{gardner} for more details). For $p\in\R$ and for $a,b\geq 0$, $\lambda\in[0,1]$ we define a $p-$average as follows:
\begin{equation}\label{p-av}
M_p(a,b,\lambda)=\left(\lambda a^p+(1-\lambda)b^p\right)^{\frac{1}{p}}.
\end{equation}
In the special cases $p=0,$ $p=+\infty$ and $p=-\infty$ we have 
$$M_0(a,b,\lambda)=a^{\lambda} b^{1-\lambda},$$
$$M_{-\infty}(a,b,\lambda)=\min(a,b),$$
$$M_{+\infty}(a,b,\lambda)=\max(a,b).$$

We say that a function $g:\R^n\rightarrow \R^+$ is $p-$concave if for every $x,y\in \R^n$ such that $g(x)g(y)>0,$ and for every $\lambda \in [0,1]$ one has
$$g(\lambda x+(1-\lambda)y)\geq M_p(g(x),g(y),\lambda).$$

We remark that $0-$concave functions are also called log-concave.

The following generalized Brunn-Minkowski inequality is well known (see e.g. Borell \cite{bor}, Gardner \cite{gardner}). Let $p\in [-\frac{1}{n},+\infty]$, and let $\mu$ be a measure on $\R^n$ with $p-$concave density $g.$ Let $$q=\frac{p}{np+1}.$$
Then the measure $\mu$ is $q-$concave on $\R^n$. That is, for every pair of Borel sets $A$ and $B$ and for every $\lambda\in[0,1]$ one has
\begin{equation}\label{genbm}
\mu(\lambda A+(1-\lambda) B)\geq M_q(\mu(A),\mu(B),\lambda).
\end{equation}

\subsection{The surface area measure, its properties and applications.}

Support hyperplane of a convex body $K$ at a point $y\in \partial K$ is a hyperplane which contains $y$ and does not contain any of the interior points of $K$. By convexity, such hyperplane exists at every point $y\in \partial K$, and is unique almost everywhere with respect to the $(n-1)$-dimensional Hausdorff measure on $\partial K$. The vector orthogonal to a support hyperplane at $y\in\partial K$ is called \emph{normal vector} at $y$; if such vector is unique it shall be denoted $n_y.$ The Gauss map $\nu_K:\partial K\rightarrow \sfe$ corresponds $y\in \partial K$ to the set of its normal vectors.

The push forward of the $(n-1)-$dimensional Hausdorff measure on $\partial K$ under the Gauss map $\nu_K$ to the sphere is called surface area measure of $K$ and is denoted by $\sigma_K.$ %In other words, for a set $A\subset \partial K,$
%$$H_{n-1}(A)=\int_{\nu_K(A)} d\sigma_K(u),$$
%where $H_{n-1}$ stands for $(n-1)-$dimensional Hausdorff measure of $A.$ 
In particular, $|\partial K|^+$ (the surface area of $K$) can be found as
$$|\partial K|^+=\int_{\sfe} d\sigma_K(u).$$

A class of strictly convex bodies whose support function is twice continuously differentiable we shall denote by $C^{2,+}$ (strict convexity means that the interior of every interval connecting a pair of points in the body is fully contained in the interior of the body). For such bodies, the Gauss map is a bijection, and the surface area measure $\sigma_K$ has a continuous density $f_K(u)$, which is called curvature function of $K.$

One can see via approximation by polytopes, that
$$\int_{\sfe} u d\sigma_K(u)=0.$$
Conversely, the following Minkowski's existence Theorem holds (see e. g. Schneider \cite{book4} or Koldobsky \cite{Kold}).
\begin{theorem}[Minkowski]\label{Minkowski}
Let $\mu$ be a measure on the sphere, not supported on any subspace, and such that  
$$\int_{\sfe} u d\mu(u)=0.$$
Then there exists a unique convex body $K$ so that $d\sigma_K(u)=d\mu(u)$ for all $u\in\sfe.$
\end{theorem}
We refer the reader to Schneider \cite{book4} for an accessible proof of Minkowski's existence theorem, and to Pogorelov \cite{Pog} for a detailed survey on the differential geometric approach to Minkowski's existence theorem, its strengthening and related results. %See Lutwak, Yang, Zhang \cite{LYZ} for the normalized $L_p-$Minkowski problem.

The support function $h_K$ of a convex body $K$, containing the origin, is defined on $\R^n$ via
$$h_K(x)=\max_{y\in K} \langle x,y\rangle.$$
Geometrically, for a unit vector $\theta$, the value of $h_K(\theta)$ represents distance to the support hyperplane of $K$ in the direction $\theta.$ Due to the fact that $h_K$ is 1-homogenous, one has
\begin{equation}\label{gradprop}
\langle \nabla h_K(u),u\rangle=h_K(u),
\end{equation}
for every $u\in \sfe$, provided that $\nabla h_K(u)$ is well-defined. In this case, $\nabla h_K(n_y)=y$ for all $y\in\partial K.$

We state a formula for a volume of a convex body $K$ with surface area measure $\sigma_K$:
\begin{equation}\label{volume}
|K|=\frac{1}{n}\int_{\sfe} h_K(u)d\sigma_K(u).
\end{equation}
The validity of this formula can be seen in the case when $K$ is a polytope and the general case follows by approximation. Moreover, for arbitrary convex bodies $K$ and $L$ one has the following:
\begin{equation}\label{mixed_volume}
V_1(K,L)=\frac{1}{n}\int_{\sfe} h_L(u)d\sigma_K(u).
\end{equation}
Another formula involving surface area measure is the so called Cauchy projection formula:
\begin{equation}\label{Cauchy}
|K|\theta^{\perp}|_{n-1}=\frac{1}{2}\int_{\sfe} |\langle u, \theta \rangle| d\sigma_K(u),
\end{equation}
where $\theta$ is an arbitrary unit vector, and $K$ is a convex symmetric body. The validity of (\ref{Cauchy}), once again, can be seen for polytopes and it follows by approximation for arbitrary convex bodies. See Koldobsky \cite{Kold} for more details about (\ref{volume}), (\ref{mixed_volume}) and (\ref{Cauchy}).

\subsection{Fourier transform on $\sfe$ and its applications to Convex geometry.}

Fourier transform in Convexity plays a very important role. See books by Koldobsky \cite{Kold}, Koldobsky, Yaskin \cite{KY}, and a survey by Koldobsky, Ryabogin, Zvavitch \cite{KRZ-2} for a detailed introduction to the subject.

The Schwartz class $\bf{S}$ is the space of complex valued rapidly decreasing infinitely differentiable functions on $\R^n$. Every locally integrable real valued function $f$ on $\R^n$ with power growth at infinity represents a distribution acting by integration:
$$\langle f,\varphi\rangle=\int_{\R^n}f(x)\varphi(x)dx,$$
for $\varphi\in\bf{S}$.

The Fourier transform $\widehat{f}$ of a distribution $f$ is defined by 
$$\langle\widehat{f},\widehat{\varphi}\rangle= (2\pi)^n\langle f,\varphi\rangle,$$ 
for every test function $\varphi\in\bf{S}$.

Let $\mu$ be a finite Borel measure on the unit sphere $\sfe$. Let $\mu_e$ be a $-(n+1)$-homogenous extension of $\mu$ to $\R^n.$ $\mu_e$ is called the extended measure of $\mu$ if for every $\varphi\in\bf{S},$
$$\langle \mu_e, \varphi \rangle=\frac{1}{2}\int_{\sfe} \langle r^{-2},\varphi(ru)\rangle d\mu(u).$$
The following geometric representation of Fourier transform on the sphere was proved by Koldobsky, Ryabogin, Zvavitch \cite{KRZ} (see also Koldobsky \cite{Kold}):
\begin{equation}\label{Fourier_scalar}
\widehat{\mu_e}(\theta)=-\frac{\pi}{2}\int_{\sfe} |\langle u, \theta\rangle|d\mu(u),
\end{equation}
for every $\theta\in \sfe.$

Note that (\ref{Cauchy}) and (\ref{Fourier_scalar}) impy that 

\begin{equation}\label{Fourier_curvature}
d\widehat{\sigma_K}(\theta)=-\pi|K|\theta^{\perp}|d\theta,
\end{equation}
where $\sigma_K$ is the surface area measure of a symmetric convex body $K$, extended to $\R^n$ with degree of homogeneity $-(n+1)$.

The following Parseval-type identity was proved by Koldobsky, Ryabogin, Zvavitch \cite{KRZ} (see also Koldobsky \cite{K98}, \cite{Kold}): for symmetric convex bodies $K, L$, so that the support function of $K$ is infinitely smooth,
\begin{equation}\label{parseval}
\int_{\sfe}\widehat{h_K}(\theta)\widehat{f_L}(\theta)=(2\pi)^n\int_{\sfe}h_K(\theta)f_L(\theta),
\end{equation}
where the Fourier transform of $h_K$ is considered with respect to its $1-$homogenous extension, and the Fourier transform of $f_L$ is considered with respect to its $-(n+1)-$homogenous extension.

By Minkowski's existence Theorem, for every symmetric convex body $L$ and for every even density $g$, not supported on a great subsphere, there exists a symmetric convex body $\tilde{L}$ such that
$$\sigma_{\mu,L}=\sigma_{\tilde{L}}.$$
Therefore, for all infinitely smooth symmetric convex bodies $K, L$ in $\R^n$, and for every even density $g$ continuous on its support, one has
\begin{equation}\label{parseval_measures}
\int_{\sfe}\widehat{h_K}(\theta)d\widehat{\sigma_{\mu,L}}(\theta)=(2\pi)^n\int_{\sfe}h_K(\theta)d\sigma_{\mu,L}(\theta),
\end{equation}
where the Fourier transform of $h_K$ is considered with respect to its $1-$homogenous extension, and the Fourier transform of $\sigma_{\mu,L}$ is considered with respect to its $-(n+1)-$homogenous extension.

Another observation is that (\ref{Fourier_scalar}) implies:

\begin{equation}\label{Fourier_scalar_flghl}
d\widehat{\sigma_{\mu,L}}(\theta)=-\frac{\pi}{2}\int_{\sfe} |\langle u, \theta\rangle| d\sigma_{\mu,L}(u),
\end{equation}
where the Fourier transform of $\sigma_{\mu,L}$ is considered with respect to its $-(n+1)-$homogenous extension. 

In particular, considering $tL$ in place of $L$ we get
\begin{equation}\label{Fourier_p_t}
\widehat{\sigma_{\mu,tL}}(\theta)=-\frac{\pi}{n}p_{\mu,L}(\theta,t),
\end{equation}
and 
\begin{equation}\label{Fourier_p}
\widehat{\int_0^1 \sigma_{\mu,tL}(\theta)dt}=-\frac{\pi}{n}P_{\mu,L}(\theta).
\end{equation}

\begin{remark}
The degree of homogeneity with which a function on the sphere is extended to $\R^n$ impacts radically its Fourier transform, and, in particular, the restriction of its Fourier transform back to the unit sphere (see more in Goodey, Yaskin, Yaskina \cite{GYY}.) We would like to emphasize the fact that the homogeneity properties of the measure $\mu$ on $\R^n$ are completely irrelevant to the study of Fourier transforms of $h_K$ and $\sigma_{\mu,K}$. In fact, we always extend $h_K$ and $\sigma_{\mu,K}$ in the most convenient way, after having already translated all the information about the underlying measure $\mu$ onto the sphere. The proof of  Theorem \ref{main}, much like the classical Shephard's problem (see \cite{KRZ}), consists of gluing  together Fourier transform and Brunn-Minkowski theory; the part which involves Fourier transform works for arbitrary measures, while the Brunn-Minkowski part is what reinforces the assumptions of concavity and homogeneity on the density of $\mu.$
\end{remark}

\subsection{Projection bodies}
Let K be an origin symmetric convex body in $\R^n$ with curvature function $f_K$. The projection body $\Pi K$ of $K$ is defined as the origin symmetric convex body in $\R^n$ whose support function in every direction is equal to the volume of the orthogonal projection of K in this direction. We extend $h_{\Pi K}$ to a homogeneous function of degree $1$ on $\R^n$. By
(\ref{Fourier_curvature}),
$$
h_{\Pi K}(\theta)=-\frac{1}{\pi} \widehat{f_K}(\theta).
$$
The curvature function of a convex body is non-negative. Therefore, $\widehat{h_{\Pi K}}\leq 0$. On the other hand, by Minkowski's existence theorem, an origin symmetric convex body $K$ in $\R^n$ is the projection body of some origin symmetric convex body if and only if there exists a measure $\mu$ on $\sfe$ so that
$$\widehat{h_K}=-\mu_e.$$

The condition that $L$ is a projection body is equivalent to $L$ being a centered zonoid (see Gardner \cite{G}). Zonoids are characterized as polar bodies of unit balls of finite dimensional sections of $L_1$.

Every origin symmetric convex body on the plane is a projection body (see Herz \cite{He}, Ferguson \cite{Fe}, Lindenstrauss \cite{Li}). It was proved by Koldobsky \cite{K1} that $p-$balls in $\R^n$ for $n\geq 3$ and $p\in [1,2]$ are not projection bodies.

\section{Mixed measures and related results}

\subsection{Mixed measures}
As an analogue of the classical mixed volume consider the following notion.
\begin{definition}\label{mixed measure}
Given sets $K$ and $L$, we define their \textbf{mixed $\mu-$measure} as follows.
$$\mu_1(K,L)=\liminf_{\epsilon\rightarrow 0} \frac{\mu(K+\epsilon L)-\mu(K)}{\epsilon}.$$
\end{definition}
We observe that in the absence of homogeneity of $\mu$, the mixed measure $\mu_1(K,L)$ is not homogenous in $K$. However, it is necessarily homogenous in $L$:
$$\mu_1(K, sL)=s\mu_1(K,L).$$
If, additionally, the measure $\mu$ is $\alpha-$homogenous, i.e.
$$\mu(tA)=t^{\alpha}\mu(A)$$
for all $t\in \R^+$ and Borel sets $A,$ then 
$$\mu_1(tK, L)=t^{\alpha-1}\mu_1(K,L).$$
\begin{definition}\label{mixed measure-volume}
We also introduce the following analogue of mixed volume:
$$V_{\mu,1}(K,L)=\int_0^1 \mu_1(tK,L) dt.$$
\end{definition}
Note that in the case of the Lebesgue measure $\lambda$ we have
$$V_{\lambda,1}(K,L)=V_1(K,L).$$
Definition \ref{mixed measure} implies that for $t\in (0,\infty)$,
\begin{equation}\label{measure-derivative}
\mu_1(tK,K)=\mu(tK)'_t;
\end{equation}
this derivative exists by monotonicity. Therefore,
\begin{equation}\label{measure-derivative-integral}
V_{\mu,1}(K,K)=\int_0^1 \mu_1(tK,K) dt=\int_0^1 \mu(tK)' dt=\mu(tK)|_0^1=\mu(K).
\end{equation}
%If measure $\mu$ is $\alpha-$homogenous then, by (\ref{measure-derivative-integral}), $\mu_1(K,K)=\alpha \mu(K)$.

Recall that we use the notation $\sigma_{\mu,K}$ for a surface area measure of a convex body $K$ with respect to a measure $\mu$ on $\R^n$. That is, for a Borel set $A\subset \sfe,$
$$\sigma_{\mu,K}(A)=\int_{\nu_K^{-1}(A)} g(x) dH_{n-1}(x),$$
where $dH_{n-1}(x)$ stands for the $(n-1)-$dimensional Hausdorff measure on $\partial K$. Following the idea from the appendix of \cite{L}, we prove the following representation for $\mu_1(K,L)$.

\begin{lemma}\label{formula-mixed-measure}
Given convex bodies $K$ and $L$ containing the origin, and a measure $\mu$ with continuous density $g$ on $\R^n$, we have 
$$\mu_1(K,L)=\int_{\sfe} h_L(u) d\sigma_{\mu,K}(u).$$
Here $h_K$ and $h_L$ are support functions of $K$ and $L$ and $\sigma_{\mu,K}$ is the surface area measure of $K$.
\end{lemma}
The proof is outlined in the Appendix (see Lemma \ref{jacapp}).

In order to provide some intuition about $\sigma_{\mu,K}$, we describe it explicitly in a couple of partial cases.

\begin{proposition}\label{sam_smooth}
If a body $K$ is $C^2-$smooth and strictly convex then its surface area measure has representation
$$d\sigma_{\mu,K}(u)=f_K(u)g(\nabla h_K(u))du.$$
\end{proposition}

\begin{proposition}\label{sam_pol}
The surface area measure of a convex polytope $P$ with respect to a measure $\mu$ has representation
$$d\sigma_{\mu,P}(u)=\sum_{i=1}^N \delta_{u_i} \mu_{n-1}(F_i)du,$$
where $u_i$, $i=1,...,N$ are the normals to the faces of the polytope, $F_i$ are the corresponding faces, and $\mu_{n-1}(F_i)$ stands for $\int_{F_i} g(x)dx$.
\end{proposition}

See the Appendix for the proofs of Propositions \ref{sam_smooth} (Proposition \ref{1ref}) and \ref{sam_pol} (Proposition \ref{2ref}).

We remark that Lemma \ref{formula-mixed-measure}, Proposition \ref{sam_smooth}, along with (\ref{Fourier_p_t}) and (\ref{parseval}) imply for all symmetric convex infinitely smooth bodies $K$ and $L$:
$$\mu_1(tK,L)=(2\pi)^{-n}\int_{\sfe} \widehat{h_L}(u) d\widehat{\sigma_{\mu,tK}}(u) du=$$
\begin{equation}\label{mu1-dual}
-\frac{\pi}{n}(2\pi)^{-n}\int_{\sfe} \widehat{h_L}(u) p_{\mu,K}(t,u) du.
\end{equation}

As an immediate corollary of Lemma \ref{formula-mixed-measure} and (\ref{measure-derivative-integral}) we derive the following expression of the measure of a $C^{2,+}$ convex body (see also \cite{CLM}).

\begin{lemma}\label{formula} Let $\mu$ be a measure with continuous density $g$. Let $K$ be a $C^{2,+}$ convex body with support function $h_K$ and curvature function $f_K$. Then
\begin{equation}\label{volume formula2}
\mu(K)=\int_{\s^{n-1}} h_K(u)f_K(u)\int_0^{1} t^{n-1} g\left(t\nabla h_K(u)\right)dt du.
\end{equation}
\end{lemma}

We outline that if the density of a measure $\mu$ on $\R^n$ is $r-$homogenous, then 
\begin{equation}\label{outl}
\mu(K)=\int_0^1 \mu_1(tK,K)dt= \mu_1(K,K)\int_0^1 t^{n+r-1}dt=\frac{1}{n+r}\mu_1(K,K).
\end{equation}

In view of (\ref{outl}), Lemma \ref{formula-mixed-measure} and Proposition \ref{sam_pol} imply the following.

\begin{proposition}\label{poly-vol}
Let $\mu$ be a measure with $r-$homogenous density $g(x)$ on $\R^n$, and consider a polytope with $N$ faces:
$$P=\{x\in\R^n:\,\,\langle x,u_i\rangle\leq \alpha_i\},$$
where $u_i\in\sfe$ and $\alpha_i>0,$ $i=1,...,N.$ Let $F_i$ be faces of $P$ orthogonal to $u_i$, $i=1,...,N.$ Then
$$\mu(P)=\frac{1}{n+r}\sum_{i=1}^N \alpha_i \mu_{n-1}(F_i),$$
where $\mu_{n-1}(F_i)$ stands for $\int_{F_i} g(x)dx$.
\end{proposition}

\subsection{Minkowski's first inequality generalized}

The main result of this subsection is the following theorem.

\begin{theorem}\label{mink_gen}
Let $\mu$ on $\R^n$ be a measure. Assume that $\mu$ is $F(t)-$concave, i.e. there exists a differentiable invertible function $F:\R^+\rightarrow \R$ such that for every $\lambda\in [0,1]$ and for every pair of Borel sets $K$ and $L$ in a certain class, we have 
\begin{equation}\label{condbm}
\mu(\lambda K+(1-\lambda) L)\geq F^{-1}\left(\lambda F(\mu(K))+(1-\lambda)F(\mu(L))\right).
\end{equation}
Then the following holds:
\begin{equation}\label{condbm-inf}
\mu_1(K,L)\geq \mu_1(K,K)+\frac{F(\mu(L))-F(\mu(K))}{F'(\mu(K))},
\end{equation}
for all $K,$ $L$ in that class.
\end{theorem}
\begin{proof} We write
$$\mu(K+\epsilon L)=\mu\left((1-\epsilon)\frac{K}{1-\epsilon}+\epsilon L\right)\geq$$ 
$$F^{-1}\left((1-\epsilon) F\left(\mu(\frac{K}{1-\epsilon})\right)+\epsilon F\left(\mu(L)\right)\right)=:G_{K,L,\mu,F}(\epsilon).$$
Note that $G_{K,L,\mu,F}(0)=\mu(K)$. Therefore, 
$$\mu_1(K,L)\geq G_{K,L,\mu,F}'(0).$$
We note that
$$\mu\left(\frac{K}{1-\epsilon}\right)'|_{\epsilon=0}=\mu_1(K,K).$$
Using the above along with standard rules of differentiation, such as 
$$(F^{-1}(a))'=\frac{1}{F'(F^{-1}(a))},$$ 
we get the statement of the Theorem. 
\end{proof}

A standard argument implies that the equality cases of the inequality (\ref{condbm-inf}) coincide with equality cases of (\ref{condbm}). We shall formulate a few corollaries of Theorem \ref{mink_gen} in some special cases.

\begin{corollary}\label{p-concave}
Let $p\geq 0$. Let $g:\R^n\rightarrow \R^+$ be a p-concave density of measure $\mu$, continuous on its support. Let $q=\frac{1}{n+\frac{1}{p}}$. Then for every pair of Borel sets $K$ and $L$ we have
$$\mu_1(K,L)\geq \mu_1(K,K)+\frac{\mu(L)^q-\mu(K)^q}{q\mu(K)^{q-1}}.$$
\end{corollary}
The corollary \ref{p-concave} follows from Theorem \ref{mink_gen} via considering $F(t)=t^q$. We also obtain the following nicer-looking corollary for measures with $p-$concave and $\frac{1}{p}-$homogenous densities. It was originally proved by E. Milman and L. Rotem \cite{MilRot}.

\begin{corollary}[E. Milman, L. Rotem]\label{p-concave-homo}
Let $p\geq 0$. Let $g:\R^n\rightarrow \R^+$ be a $p$-concave $\frac{1}{p}-$homogenous density of measure $\mu$. Let $q=\frac{1}{n+\frac{1}{p}}$. Then for every pair of Borel sets $K$ and $L$ we have
\begin{equation}\label{mink-gen-1}
\mu_1(K,L)\geq \frac{1}{q}\mu(K)^{1-q}\mu(L)^q,
\end{equation}
and
\begin{equation}\label{mink-gen-2}
V_{\mu,1}(K,L)\geq \mu(K)^{1-q}\mu(L)^q.
\end{equation}
\end{corollary}
\begin{proof} Note that if $g$ is $\frac{1}{p}-$homogenous then $\mu$ is an $(n+\frac{1}{p})=\frac{1}{q}-$homogenous measure. Therefore, 
\begin{equation}\label{homo}
V_{\mu,1}(K,L)=\int_0^1 \mu_1(tK,L)dt=\mu_1(K,L)\int_0^1
 t^{\frac{1}{q}-1} dt=q\mu_1(K,L),
 \end{equation}
 and in particular
\begin{equation}\label{homo-K}
\mu(K)=q\mu_1(K,K)
\end{equation}

Corollary \ref{p-concave} together with (\ref{homo-K}) implies (\ref{mink-gen-1}). Also, (\ref{mink-gen-1}) together with (\ref{homo}) implies (\ref{mink-gen-2}). 
\end{proof}

Recall that a measure $\mu$ is called log-concave if for all Borel sets $K$ and $L$,
$$\mu(\lambda K + (1-\lambda)L)\geq \mu(K)^{\lambda}\mu(L)^{1-\lambda}.$$
Applying Theorem \ref{mink_gen} with $F(t)=\log t$ (as $\log t$ is an increasing function), we get the following corollary.
\begin{corollary}\label{log-concave}
Let measure $\mu$ be log-concave. Then for every pair of Borel sets $K$ and $L$ we have
$$\mu_1(K,L)\geq \mu_1(K,K)+\mu(K)\log\frac{\mu(L)}{\mu(K)}.$$
\end{corollary}

In particular, the following isoperimetric-type result follows from Theorem \ref{mink_gen}. 

\begin{proposition}\label{isoper}
Let a measure $\mu$ be log-concave. Then for every pair of Borel sets $K$ and $L$ such that $\mu(K)=\mu(L)$, one has
$$\mu_1(K,L)\geq \mu_1(K,K).$$
\end{proposition}
For example, if $\gamma$ is the standard Gaussian measure $\gamma$ (that is, the measure with density $\frac{1}{\sqrt{2\pi}^n} e^{-\frac{|x|^2}{2}}$), and $K$ is a convex set containing the origin, then the expression
$$\int_{\partial K} \langle y, \nu_L(y)\rangle  e^{-\frac{|y|^2}{2}} d\sigma(y)$$
is minimized when $L=K,$ where $L$ is such convex region that $\gamma(K)=\gamma(L)$, and $\nu_L$ is it Gauss map. 

Another strengthening of Corollary \ref{log-concave} in the case of the standard Gaussian measure is possible to obtain using Ehrhard's inequality (see Ehrhard \cite{E}, Borell \cite{bor-erch}). Recall the notation
$$\psi(a)=\frac{1}{\sqrt{2\pi}}\int_{-\infty}^a e^{-\frac{t^2}{2}} dt.$$
It was shown by Ehrhard (for convex sets), and further extended by Borell, that for every pair of Borel sets $K$ and $L$ and for every $\lambda\in [0,1]$ we have
$$\psi^{-1}\left(\gamma(\lambda K+(1-\lambda) L)\right)\geq \lambda \psi^{-1}(\gamma(K))+(1-\lambda)\psi^{-1}(\gamma(L)).$$
Hence the next Corollary follows.
\begin{corollary}\label{gaussian}
For the standard Gaussian measure $\gamma$ and for every pair of convex sets $K$ and $L$ we have
$$\gamma_1(K,L)\geq \gamma_1(K,K)+e^{-\frac{\psi^{-1}(\gamma(K))^2}{2}}\left(\psi^{-1}(\gamma(L))-\psi^{-1}(\gamma(K))\right).$$
\end{corollary}
To obtain this corollary we use the fact that $\psi$ is an increasing function and the relation
$$\psi^{-1}(a)'=e^{\frac{\psi^{-1}(a)^2}{2}}.$$

\section{Extension of the Minkowski's existence theorem.}

This section is dedicated to proving an extension of Minkowski's existence theorem. We use ideas from the proof of the classical Minkowski's existence theorem (see Schneider \cite{book4}). 

First, we state a definition.

\begin{definition}
For a measure $\mu$ on $\R^n$, we say that a convex body $K$ in $\R^n$ with particular properties is \textbf{$\mu-$unique} if every pair of convex bodies with said properties coincides up to a set of $\mu$-measure zero.
\end{definition}

\begin{theorem}\label{mink-existence}
Let $\mu$ on $\R^n$ be a measure and $g(x)$ be its even $r-homogenous$, continuous on its support density for some $r\geq 0$, such that a restriction of $g$ on some half space is $p-concave$ for $p\geq 0$. Let $\varphi$ be an arbitrary even measure on $\sfe$, not supported on any great subsphere, such that $supp(\varphi)\subset int(supp(g))\cap \sfe$. Then there exists a $\mu-$unique convex body $K$ in $\R^n$ such that 
$$d\sigma_{K,\mu}(u)=d\varphi(u).$$
%Moreover, if restriction of $g$ on a half space is $p-concave$ for $p>0$ then such convex body $K$ is unique.
\end{theorem}
The existence part of Theorem \ref{mink-existence} follows by approximation from the lemma below. We remark, that for an $(n-1)-$dimensional surface $F$, the notation $\mu_{n-1}(F)$ stands for 
$$\mu_{n-1}(F)=\int_F g(x)dx,$$
where $g(x)$ is the density of $\mu,$ and $dx$ is the area element on $F.$
\begin{lemma}\label{mink-existence-polytopes}
Let $\mu$ on $\R^n$ be a measure and $g(x)$ be its even $r-homogenous$ continuous on its support density for some $r>-n$. Let $N\geq 2n$ be an even integer. Let $u_1,...,u_N$ be unit vectors spanning the $\R^n$, $u_i\in int(supp(g))$, such that $u_i=-u_{\frac{N}{2}+i}$. Let $f_1,...,f_N$ be arbitrary positive numbers such that $f_i=f_{\frac{N}{2}+i}$.

Then there exist positive $\alpha_1,...,\alpha_N$ such that the convex polytope 
$$P=\cap _{i=1}^N \{|\langle x,u_i\rangle|\leq \alpha_i\}$$ 
with faces $F(u_1),...,F(u_N)$ satisfies
$$\mu_{n-1}(F(u_i))=f_i.$$
Moreover, if restriction of $g$ on a half space is $p-concave$ for $p\geq 0$ then such polytope $P$ is $\mu-$unique.
\end{lemma}
\begin{proof} For a vector $A=(\alpha_1,...,\alpha_N)\in\R^N$ we shall consider a polytope
$$P(A)=\cap_{i=1}^N\{x\in\R^n\,:\,|\langle x,u_i\rangle| \leq \alpha_i\}.$$
Consider a set $M\subset \R^N$ defined as follows:
$$M:=\{A\in \R^N\,:\, \mu(P(A))\geq 1\}.$$
Note that $M\subset \{A:\,\, \alpha_i\geq 0\,\,\forall i=1,...,N\}$. It is nonempty since the measure is unbounded. As the set $M$ is closed, and $f_i>0,$ the linear functional
$$\varphi(A)=\frac{1}{n+r}\sum_{i=1}^N f_i \alpha_i$$ 
attains its minimum on $M.$ Let $A^*=(\alpha_1^*,...,\alpha_N^*)$ be the minimizing point, $P^*=P(A^*)$, and let $F_i^*$ stand for the facet of $P^*$ orthogonal to $u_i$. Denote the value of the minimum $\varphi(A^*)=m^{n+r-1}$. 

We show that $m P^*$ is the polytope which solves the problem. Indeed, consider hyperplanes
$$H_1=\{A\in \R^N\,:\, \frac{1}{n+r}\sum_{i=1}^N f_i \alpha_i=m^{n+r-1}\},$$
$$H_2=\{A\in \R^N\,:\, \frac{1}{n+r}\sum_{i=1}^N \mu_{n-1}(F_i^*) \alpha_i=1\}.$$
Note that all $\alpha_i^*>0$. Thus, by Proposition \ref{poly-vol}, 
$$\mu(P^*)=\frac{1}{n+r}\sum_{i=1}^N \alpha_i^*\mu_{n-1}(F_i^*).$$
On the other hand, the linear functional $\varphi$ attains its minimum on the boundary of $M$, and hence 
\begin{equation}\label{mu^*}
\mu(P^*)=1. 
\end{equation}
We conclude that $A^*\in H_1\cap H_2.$ 

Observe that $H_1\cap int(M)=\varnothing$, as otherwise $A^*$ would not be the minimum. Consider a vector $A\in H_1$ different from $A^*$. For any $\lambda\in [0,1]$, the vector $\lambda A^*+(1-\lambda) A\in H_1$, and hence 
$$\mu(P(\lambda A^*+(1-\lambda) A))\leq 1.$$
Note also that
$$\lambda P(A^*)+(1-\lambda) P(A)\subset P(\lambda A^*+(1-\lambda) A),$$
and thus
\begin{equation}\label{muleq}
\mu(\lambda P^*+(1-\lambda) P(A))\leq 1.
\end{equation}
Therefore, by homogeneity of $\mu$, (\ref{mu^*}) and (\ref{muleq}),
$$\mu_1(P^*,P(A))=\liminf_{\epsilon\rightarrow 0}\frac{\mu(P^*+\epsilon P(A))-\mu(P^*)}{\epsilon}=$$
$$\liminf_{\epsilon\rightarrow 0}\frac{(1+\epsilon)^{n+r}\mu(\frac{1}{1+\epsilon}P^*+\frac{\epsilon}{1+\epsilon} P(A))-1}{\epsilon}\leq \liminf_{\epsilon\rightarrow 0}\frac{(1+\epsilon)^{n+r}-1}{\epsilon}=n+r.$$

On the other hand, if $\alpha_i>0$, by Proposition \ref{sam_pol} and Lemma \ref{formula-mixed-measure} we have 
$$\mu_1(P^*,P(A))=\sum_{i=1}^N \alpha_i \mu_{n-1}(F^*_i),$$
and hence
\begin{equation}\label{muleq-1}
\frac{1}{n+r}\sum_{i=1}^N \alpha_i \mu_{n-1}(F_i^*)\leq 1.
\end{equation}
Therefore, there exists an open subset of $H_1$,
$$U:=H_1\cap \{A\in\R^N\,:\,\alpha_i>0\},$$ 
which is fully contained in the half space 
$$H_2^-=\{A\in \R^N\,:\, \frac{1}{n+r}\sum_{i=1}^N \mu_{n-1}(F_i^*) \alpha_i\leq1\},$$
and, in addition, the interior of $U$ contains $A^*\in H_1\cap H_2$.  This implies that $H_1=H_2.$

Therefore, 
$$\mu_{n-1}(F_i^*)m^{n+r-1}=f_i.$$
Using homogeneity of $g$ once again, we conclude that the polytope 
$$m P^*=\cap_{i=1}^N \{x\in \R^n\,:\, \langle x,u_i\rangle \leq \beta_i\},$$ 
with $\beta_i=m \alpha_i^*$, satisfies the conclusion of the Lemma. 

The uniqueness part follows in the same manner as in subsection \ref{un} for all convex bodies, therefore we skip the argument here.
\end{proof}

We remark that no concavity was necessary to prove the existence part for polytopes; however, it is used in the proof for uniqueness, and it is used in the approximation argument below. 

\subsection{Proof of the uniqueness part of Theorem \ref{mink-existence}.}\label{un}

\begin{proof} Let $\tilde{\mu}$ be measure with density $\tilde{g}(u)=g(u)1_{\{\langle u,v\rangle>0\}}$, for some unit vector $v$, such that $\tilde{g}$ is $p-$concave and $\frac{1}{p}-$homogenous on its support for some $p\geq 0$ (assumptions of the Theorem along with Proposition \ref{p-concave-proposition} of the appendix allow us to select such vector). Fix $q=\frac{p}{np+1}$. Assume that there exist two symmetric convex bodies $K$ and $L$ such that
\begin{equation}\label{condition1111}
d\sigma_{\mu,K}(u)=d\sigma_{\mu,L}(u)
\end{equation}
for all $u\in\sfe$. Observe that
$$\mu_1(K,L)=\int_{\sfe} h_K(u) d\sigma_{\mu,L}(u)=$$
$$\int_{\sfe} h_K(u) d\sigma_{\mu,K}(u)=\mu_1(K,K)=\frac{1}{q} \mu(K).$$
By symmetry of $K$ and $L$, it implies that 
$$\tilde{\mu}_1(K,L)=\frac{1}{q} \tilde{\mu}(K).$$
By Corollary \ref{p-concave-homo}, 
\begin{equation}\label{mink-torefer}
\frac{1}{q}\tilde{\mu}(K)=\tilde{\mu}_1(K,L)\geq \frac{1}{q}\tilde{\mu}(K)^{1-q}\tilde{\mu}(L)^q,
\end{equation}
and hence $\tilde{\mu}(K)\geq\tilde{\mu}(L)$. Analogously, by considering $\tilde{\mu}_1(L,K)$, we get that $\tilde{\mu}(K)\leq \tilde{\mu}(L)$. Hence $\tilde{\mu}(K)=\tilde{\mu}(L)$, and hence there is equality in (\ref{mink-torefer}). Milman and Rotem (\cite{MilRot} Corollary 2.17) proved, using the results from Dubuc \cite{Dub}, that in this case $K$ and $L$ have to coincide up to a dilation and a shift on the support of $\tilde{\mu}$. As we assume that $K$ and $L$ are symmetric, we get that $K=aL$ for some $a>0$ almost everywhere with respect to $\tilde{\mu}$. But as $g$ is $\frac{1}{p}$-homogenous, we have 
$$
d\sigma_{\mu,K}(u)=d\sigma_{\mu,aL}(u)=a^{n+\frac{1}{p}-1}d\sigma_{\mu,L}(u),
$$
and hence by (\ref{condition1111}), $a=1$. Which means that $K=L$ $\mu$-almost everywhere. 
\end{proof}

\subsection{Proof of the existence part of Theorem \ref{mink-existence}.}

\begin{proof}
We shall use Lemma \ref{mink-existence-polytopes} and argue by approximation. Let $d\varphi(u)$ be an even measure on $\sfe$. For a positive integer $k$, consider a symmetric partition of $\sfe\cap supp(\varphi)$ into disjoint sets $A_i$, $i=1,...,2N$ with spherically convex closures of diameters at most $\frac{1}{k}$ (recall that a subset of the sphere is called spherically convex if the geodesic interval connecting any pair of points in the set is fully contained in this set). Consider the vector
$$c_i=\frac{1}{\varphi(A_i)}\int_{A_i} ud\varphi(u).$$
Note that $c_i\neq 0$. Select $u_i\in\sfe$ and $f_i\in\R^+$ to be such that $c_i=f_i u_i.$ Note that $u_i\in int(A_i)$. Therefore, for every $u\in A_i$, $|u-u_i|\leq \frac{1}{k}$, and hence
\begin{equation}\label{f_i_est}
1-\frac{1}{k}\leq f_i\leq 1.
\end{equation}

According to Lemma \ref{mink-existence-polytopes}, there exists a polytope 
$$P_k=\{x\in\R^n:\, |\langle x,u_i\rangle|\leq \alpha_i\}$$ 
with faces $F_{P_K}$, such that 
$$\mu_{n-1}(F_{P_K}(u_i))=\int_{A_i} \varphi(u)du.$$
Consider a measure $\varphi_k$ on $\sfe$ such that for every Borel set $\Omega\subset \sfe,$
$$\varphi_k(\Omega)=\sum_{u_i\in\Omega}\mu_{n-1}(F_{P_K}(u_i)).$$
Consider a bounded Lipschitz function $a(u)$ on $\sfe.$  Observe that
$$\left|\int_{\sfe} a(u)d\varphi(u)-\int_{\sfe} a(u)d\varphi_k(u)\right|\leq \sum \int_{A_i}|a(u)-a(u_i) f_i| d\varphi(u).$$
Observe as well, that by (\ref{f_i_est}),
$$|a(u)-f_i a(u_i)|\leq |a(u_i)-f_i a(u_i)|+|a(u)-a(u_i)|\leq$$ 
$$\frac{1}{k}||a||_{Lip}+||a||_{\infty}|1-f_i|\leq \frac{1}{k}(||a||_{Lip}+||a||_{\infty})\rightarrow_{k\rightarrow \infty} 0.$$
%since the diameter of $A_i$ is bounded by $\frac{1}{k}$. 
Thus $\varphi_k\rightarrow \varphi$ weakly, as $k$ tends to infinity.

It remains to show that all the polytopes $P_k$ are bounded on the support of $\mu$: then, by Blaschke selection theorem (see \cite{book4}, Theorem 1.8.6), applied on the support of $\mu$, there exists a subsequence of $\{P_k\}$ which converges to some convex body $P$ in Hausdorff metric. Then $\sigma_{\mu,P_k}\rightarrow \sigma_{\mu,P}$ weakly (see Proposition \ref{app} from the appendix), and hence, by the uniqueness of the weak limit, we have $d\sigma_{\mu,P}(u)=d\varphi(u)$.

To show the boundedness, observe first that $\mu^+(\partial P_k)=\int_{\sfe} \varphi(u)du=:\tilde{C}_{\varphi}$, where $\mu^+(\partial P_k)$ stands for $\mu_1(P_k,B_2^n)$. 

Let $\tilde{g}$ be the restriction of $g$ to a half space where it is $p-concave.$ By Corollary \ref{p-concave-homo}, 
$$\tilde{\mu}(P_k)\leq \left(q\tilde{\mu}(B_2^n)^{-q}\tilde{\mu}^+(\partial P_k)\right)^{\frac{1}{1-q}},$$
and hence, by symmetry of $P_k,$
\begin{equation}\label{a1}
\mu(P_k)\leq \left(q\mu(B_2^n)^{-q}\mu^+(\partial P_k)\right)^{\frac{1}{1-q}}\leq C_{\mu,\varphi}.
\end{equation}
Here $q=\frac{p}{np+1}$, and $C_{\mu,\varphi}$ depends only on the measures $\mu$ and $\varphi.$ On the other hand, for any $x\in P_k$ we have
$$h_{P_k}(u)\geq \langle u,x\rangle^+=|x|\langle u,v\rangle^+,$$
where $v\in \sfe$ is such that $x=|x|v,$ and $\langle u,x\rangle^+$ stands for the positive part of $\langle u,x\rangle$. We note that for $k$ large enough,
$$\int_{\sfe} \langle u,v\rangle^+ d\varphi_k(u)\geq \frac{1}{2}\int_{\sfe} \langle u,v\rangle^+ d\varphi(u)=:C_{\varphi}>0,$$
where $C_{\varphi}>0$ is a positive constant depending on $\varphi$ only. Therefore,
\begin{equation}\label{a2}
\mu(P_k)=\frac{1}{n+r}\int_{\sfe} h_{P_K}(u) d\varphi_k(u) \geq |x| C_{\varphi}.
\end{equation}
By (\ref{a1}) and (\ref{a2}), $|x|\leq \frac{C_{\mu,\varphi}}{C_{\varphi}}$. As $x$ was an arbitrary point from $P_k$, we conclude that the sequence $\{P_k\}$ is indeed uniformly bounded. 
\end{proof}

\section{Applications to the questions about uniqueness.}\label{logmink}

\subsection{An extension of Aleksandrov's theorem.}

\begin{theorem}\label{aleks}
Let $\mu$ be a measure with density with positive degree of concavity and positive degree of homogeneity. Let $K$ and $L$ be symmetric convex bodies such that in every direction $\theta$, $P_{\mu, K}(\theta)=P_{\mu,L}(\theta)$. Then $K=L$ $\mu-$almost everywhere.
\end{theorem}
\begin{proof} Given $g(x)$ on $\R^n,$ the density of $\mu,$ let $\tilde{\mu}$ on $\R^n$ be the measure with density $\tilde{g}(x)=\frac{g(x)+g(-x)}{2}$. Recall that by (\ref{Fourier_p_t}),
$$
d\widehat{\sigma_{\tilde{\mu},K}}(\theta)=-C(\mu)\frac{\pi}{n}P_{\tilde{\mu},K}(\theta)
$$
and
$$
d\widehat{\sigma_{\tilde{\mu},L}}(\theta)=-C(\mu)\frac{\pi}{n}P_{\tilde{\mu},L}(\theta),
$$
where $C(\mu)$ depends only on the dimension and the degree of homogeneity of $\mu$, and the Fourier transform is considered with respect to $-(n+1)-$homogenous extensions of $\sigma_{\tilde{\mu},K}$ and $\sigma_{\tilde{\mu},L}$. 

Note that $P_{\mu, K}(\theta)=P_{\mu,L}(\theta)$ implies $P_{\tilde{\mu}, K}(\theta)=P_{\tilde{\mu},L}(\theta)$ for every $\theta$. By Fourier inversion formula, we get that $\sigma_{\tilde{\mu},K}=\sigma_{\tilde{\mu},L}$ everywhere on the sphere. By Theorem \ref{mink-existence-intro} we conclude that $K$ and $L$ coincide up to a set of $\mu$-measure zero. 
\end{proof}

\subsection{Uniqueness of solutions for certain PDE's in the class of support functions.}

\begin{proposition}\label{logbm}
Let $K$ and $L$ be two symmetric $C^{2,+}$ convex bodies in $\R^n$ with support functions $h_K$ and $h_L$ and curvature functions $f_K$ and $f_L$ such that 
$$\frac{\partial h_K(u)}{\partial x_1} f_K(u)=\frac{\partial h_L(u)}{\partial x_1} f_L(u)$$
for every $u\in\sfe.$ Then $K=L.$
\end{proposition}
\begin{proof} Let $g:\R^n\rightarrow \R^+$ be given via 
$$g(x)=|x_1|.$$
Then, for every $x\in \R^n,$
$$
g(\nabla h_K)=\left|\frac{\partial h_K(u)}{\partial x_1}\right|. 
$$
By the symmetry, the Proposition \ref{sam_smooth} and the condition of the Corollary,
\begin{equation}\label{conevolappl}
\sigma_{\mu,K}=f_K(u) g(\nabla h_K(u))=f_L(u) g(\nabla h_L(u))=\sigma_{\mu,L}
\end{equation}
for every $u\in\sfe.$
Observe that the restriction of $g$ onto $\{x\in \R^n:\,x_1>0\}$ is $1-$homogenous and $1-$concave. Therefore, it satisfies the condition of theorem \ref{mink-existence}, and thus, by (\ref{conevolappl}), $K=L$ $\mu$-almost everywhere. In this case it means that $K=L$ coincide almost everywhere with respect to Lebesgue measure, and as they are also convex bodies, it means that $K=L$. 
\end{proof}

We remark that the curvature function $f_K$ can be written in the Aleksandrov's form as $det(\delta_{ij} h+h_{ij})$, where $h$ is the support function of $K$, $h_{ij}$ are derivatives of it taken with respect to an orthonormal frame on $\sfe$, and $\delta_{ij}$ is the usual Kroneker symbol. Therefore, Proposition \ref{logbm} implies that a PDE
$$\frac{\partial h}{\partial x_1} det(\delta_{ij} h+h_{ij})=F$$
has a unique solution in the class of even support functions of convex bodies. The existence of such solution for even continuous function $F$ which is not supported on any great subsphere can be derived from Theorem \ref{mink-existence-intro}.

\begin{remark}
Observe that
$$\frac{\partial (h_K(u)f_K(u))}{\partial x_1}=\frac{\partial h_K(u)}{\partial x_1} f_K(u)+\frac{\partial f_K(u)}{\partial x_1} h_K(u).$$
Hence, by Proposition \ref{logbm}, the following pair of conditions guarantee equality of smooth symmetric sets $K$ and $L$:
\begin{enumerate}
\item $h_K (u)f_K(u)=h_K (u)f_K(u)$ at every $u\in\sfe$;\\
\item $\frac{\partial f_K(u)}{\partial x_1} h_K(u)=\frac{\partial f_L(u)}{\partial x_1} h_L(u)$ at every $u\in\sfe$.
\end{enumerate}
\end{remark}

\begin{remark}\label{logbm-remark}
Instead of requiring the condition of Proposition \ref{logbm} it is in fact enough to require that there exists a vector $v$ such that for every $u\in\sfe,$
$$f_K(u)\langle\nabla h_K(u),v\rangle=f_L(u)\langle\nabla h_L(u),v\rangle.$$
In this case we still conclude that $K=L.$
\end{remark}

\begin{remark}
The Log-Minkowski problem (see e. g. B\"or\"oczky, Lutwak, Yang, Zhang \cite{BLYZ}, \cite{BLYZ-1}, \cite{BLYZ-2}, Lutwak, Yang, Zhang \cite{LYZ}, Lutwak, Oliker \cite{LO}, Stancu \cite{St}, Huang, Liu, Xu \cite{HLX}) asks whether a symmetric convex body $K$ is uniquely defined by its cone volume measure $\frac{1}{n}h_K(u) f_K(u)$, where $u\in\sfe$.

Suppose that symmetric convex bodies $K$ and $L$ satisfy
\begin{equation}\label{cv}
h_K(u)f_K(u)=h_L(u)f_L(u),
\end{equation}
for every $u\in \sfe$. Consider a vector field 
$$a(u)=\nabla h_K(u) f_K(u)-\nabla h_L(u) f_L(u).$$
Note that by (\ref{gradprop}), (\ref{cv}) is equivalent to the fact that $a(u)$ is a tangent field, that is $a(u)\perp u$. 

In view of Corollary \ref{logbm}, unique determination of a smooth convex body would follow if one could show that in fact $a(u)$ has to be identically zero. Moreover, in view of the previous remark it would suffice to show that there exists a vector $v\in \sfe$ such that $\langle a(u), v\rangle=0$ for all $u\in\sfe.$ 
\end{remark}

\section{Extensions of the solution to Shephard's problem.}
We shall follow the scheme of the proof for the classical Shephard problem (see Koldobsky \cite{Kold}), which suggests glueing together harmonic-analytic results with the Brunn-Minkowski theory.
\subsection{General preparatory lemmas.}
To prove Theorem \ref{main}, we first need the following Lemma.
\begin{lemma}\label{shephard_general}
Let $\mu$ be a measure with density $g$ continuous on its support, and let $K, L$ be symmetric convex bodies. Assume additionally that $L$ is a projection body. Assume that for a given $t\in[0,1]$ and for every $\theta\in\sfe$ we have 
$$p_{\mu, K}(\theta,t)\leq p_{\mu, L}(\theta,t).$$
Then 
$$\mu_1(tK,L)\leq \mu_1(tL,L).$$
\end{lemma}
\begin{proof}  Without loss of generality we may assume that $K$ and $L$ are infinitely smooth strictly convex bodies; the general case then follows via standard approximation argument (see, e.g., Koldobsky \cite{Kold} Section 8). 

Consider a symmetrization of $\mu$. Let $\tilde{\mu}$ be the measure with density
$$\tilde{g}(x)=\frac{g(x)+g(-x)}{2}.$$
Since $K$ and $L$ are symmetric, we have for all $\theta\in\sfe$ and $t\in[0,1]$:
$$p_{\tilde{\mu}, K}(\theta,t)=p_{\mu, K}(\theta,t);$$
$$p_{\tilde{\mu}, L}(\theta,t)=p_{\mu, L}(\theta,t),$$
and hence
\begin{equation}\label{111111}
p_{\tilde{\mu}, K}(\theta,t)\leq p_{\tilde{\mu}, L}(\theta,t).
\end{equation}
Assume for a moment that $K$ and $L$ are strictly convex and infinitely smooth. By (\ref{Fourier_p_t}),
$$\widehat{\sigma_{\mu,tL}}(\theta)=-\frac{\pi}{n}p_{\mu,L}(\theta,t).$$
Hence, by Proposition \ref{sam_smooth},
$$p_{\tilde{\mu}, K}(\theta,t)=-\frac{n}{\pi}\widehat{f_{tK}\tilde{g}(\nabla h_{tK})}(\theta).$$
By (\ref{111111}), we get
$$\widehat{f_{tK}\tilde{g}(\nabla h_{tK})}(\theta)\geq \widehat{f_{tL}\tilde{g}(\nabla h_{tL})}(\theta),$$
for every $\theta\in\sfe$ and for every $t\in [0,1].$ As $L$ is a projection body, we have $\widehat{h_L}(\theta)\leq 0$. Thus
\begin{equation}\label{toint}
\widehat{h_L}(\theta)\widehat{f_{tK}\tilde{g}(\nabla h_{tK})}(\theta)\leq \widehat{h_L}(\theta)\widehat{f_{tL}\tilde{g}(\nabla h_{tL})}(\theta),
\end{equation}
for every $\theta\in\sfe$ and for every $t\in [0,1].$ Integrating (\ref{toint}) over the unit sphere, and applying Parseval's identity (\ref{parseval}) on both sides of the inequality, we get
\begin{equation}\label{fin}
\int_{\sfe}h_L(\theta)f_{tK}(\theta)\tilde{g}(\nabla h_{tK}(\theta))d\theta\leq \int_{\sfe}h_L(\theta)f_{tL}(\theta)\tilde{g}(\nabla h_{tL}(\theta))d\theta.
\end{equation}
Lemma \ref{formula-mixed-measure} applied along with (\ref{fin}) implies that
$$\tilde{\mu}_1(tK,L)\leq \tilde{\mu}_1(tL,L).$$
Using symmetry of $K$ and $L$ once again, we note that
$$\tilde{\mu}_1(tK,L)=\mu_1(tK,L);$$
$$\tilde{\mu}_1(tL,L)=\mu_1(tL,L),$$
and the lemma follows.
\end{proof}

Via the same scheme as above, invoking Lemma \ref{formula} along with the fact that $V_{\mu,1}(L,L)=\mu(L)$, we get the following 

\begin{lemma}\label{shephard_general-integrated}
Let $\mu$ be a measure with density $g$ continuous on its support, and let $K, L$ be symmetric convex bodies. Assume additionally that $L$ is a projection body. Assume that for every $\theta\in\sfe$ we have 
$$P_{\mu, K}(\theta)\leq P_{\mu, L}(\theta).$$
Then 
$$V_{\mu,1}(K,L)\leq \mu(L).$$
\end{lemma}

\subsection{Proof of the Theorem \ref{main}.}
\begin{proof}
As is shown in Proposition \ref{p-concave-proposition} of the Appendix, if a non-negative function has a positive degree of homogeneity and a positive degree of concavity, then there exists $p\geq 0$ such that $g$ is $p-$concave and $\frac{1}{p}-$homogenous. Additionally, such function is necessarily supported on a convex cone. 

The assumptions of the Theorem allow us to apply Lemma \ref{shephard_general-integrated} and obtain:
\begin{equation}\label{part1}
V_{1,\mu}(K,L)\leq \mu(L).
\end{equation}
On the other hand, we apply part (\ref{mink-gen-2}) of Corollary \ref{p-concave-homo} and write
$$\mu(L)\geq V_{1,\mu}(K,L)\geq \mu(K)^{1-q}\mu(L)^q,$$
where $q=\frac{p}{np+1}$. Hence $\mu(L)\geq \mu(K)$.
\end{proof}

\begin{remark}
Theorem \ref{main} does not hold for all measures. Indeed, consider measure $\mu$ with density $1_{B_2^n}$ and convex bodies $L=r B_2^n$, $K=R B_2^n$ such that $r\leq 1\leq R$ and $R\geq r^{-\frac{1}{n-1}}$. Then $P_{\mu,K}(\theta)\leq P_{\mu, L}(\theta)$ for all $\theta\in\sfe$ but $\mu(K)\geq\mu(L)$. However, requiring the inequality $p_{\mu,K}(\theta,t)\leq p_{\mu, L}(\theta, t)$ for all $\theta\in\sfe$ and for all $t\in[0,1]$ may suffice to conclude that $\mu(K)\leq\mu(L)$ for a wide class of measures with some basic concavity properties.
\end{remark}

\subsection{A general statement}
Finally, we present a measure comparison-type result for a more general class of measures. It may prove useful for considering this problem in greater generality.

\begin{proposition}
Let $\mu$ be a measure on $\R^n$ with density continuous on its support. Suppose that $\mu$ is $F(t)-$concave for some invertible $C^1$ function $F:\R^+\rightarrow \R.$ Let $K$ and $L$ be convex symmetric bodies, and let $L$ in addition be a projection body. Assume that for every $\theta\in \sfe$ and for every $t\in [0,1]$ we have
$$p_{\mu,L}(\theta,t)\geq p_{\mu,K}(\theta,t).$$
Then
\begin{enumerate}[label = (\roman*)\,]
\item $\displaystyle \mu(L)\geq \mu(K)+\int_0^1 \frac{F(\mu(tL))-F(\mu(tK))}{tF'(\mu(tK))} dt;$\\
\item $\displaystyle \mu(L)\geq \mu(K)+\int_0^1 \left[\mu(tL)-\mu(tK)+\frac{F(\mu(tL))-F(\mu(tK))}{F'(\mu(tK))}\right]dt$.\end{enumerate}
\end{proposition}

\begin{proof} By Lemma \ref{shephard_general}, we get that $\mu_1(tK,L)\leq \mu_1(tL,L)$ for every $t\in [0,1]$, and therefore
\begin{equation}\label{1}
\mu_1(tK,tL)=t\mu_1(tK,L)\geq t\mu_1(tL,L)=\mu_1(tL,tL).
\end{equation}
Applying (\ref{1}) along with Theorem \ref{mink_gen} we get
\begin{equation}\label{eq}
t\mu_1(tL,L)\geq t\mu_1(tK,K)+\frac{F(\mu(tL))-F(\mu(tK))}{F'(\mu(tK))}.
\end{equation}
After dividing both sides by $t$ and integrating we get
\begin{equation}\label{rfr}
\int_0^1\mu_1(tL,L)dt\geq \int_0^1\mu_1(tK,K)dt+ \int_0^1\frac{F(\mu(tL))-F(\mu(tK))}{tF'(\mu(tK))}dt,
\end{equation}
hence (i) follows from (\ref{measure-derivative-integral}) and (\ref{rfr}).

Next, we integrate by parts:
\begin{equation}\label{11111}
\int_0^1 t\mu_1(tL,L)dt=\mu(L)-\int_0^1 \mu(tL)dt.
\end{equation}
Thus (\ref{eq}) and (\ref{11111}) imply (ii). 
\end{proof}

\subsection{Proof of Theorem \ref{main-2}.}
\begin{proof}
Without loss of generality we may assume that the boundary of $L$ is infinitely smooth (see the approximation argument in Koldobsky \cite{Kold}, Section 8). Inasmuch as $L$ is not a projection body we have that $\widehat{h_L}$ is positive on an open set $\Omega\subset \sfe$; recall as well that, per our assumptions, the curvature function $f_L$ is positive everywhere on the sphere, and $L$ is symmetric. Let $v:\sfe\rightarrow \R$ be a non-negative infinitely smooth even function supported on $\Omega.$ Let $\tilde{g}(x)$ be the restriction of $g(x)$ on the half space where is has positive homogeneity, and let $\tilde{\mu}$ be the measure with density $\tilde{g}$. Since we assume that $g$ is supported on the whole space, $\tilde{g}$ is fully supported on a half space.

Define a symmetric convex body $K$ via the relation
\begin{equation}\label{equation}
d\sigma_{\mu,K}(u)=d\sigma_{\mu,L}(u)-\epsilon \widehat{v}(u)
\end{equation}
for every $u\in\sfe$. Here $\epsilon>0$ is chosen small enough so that the right hand side of (\ref{equation}) stays non-negative. Theorem \ref{mink-existence} guarantees that such convex body exists. Applying Fourier transform to $-(n+1)$-homogenous extensions of both sides of (\ref{equation}), we get
$$
-\frac{\pi}{n q}P_{\tilde{\mu},K}(\theta)=-\frac{\pi}{n q}P_{\tilde{\mu},L}(\theta)-\epsilon v(\theta),
$$
and hence, by symmetry of $K$ and $L$,
\begin{equation}\label{p_comp}
-P_{\mu,K}(\theta)=-P_{\mu,L}(\theta)-\frac{n q}{\pi}\epsilon v(\theta).
\end{equation}
Recall that
$$V_{\mu,1}(K,L)=\int_0^1\int_{\sfe} h_L(u) f_{tK}(u) g(\nabla h_{tK}(u)) du dt,$$
and that $P_{\mu,K}(\theta)$ is the Fourier transform of the $-(n+1)$-homogenous extension of 
$$-\frac{\pi}{n}\int_0^1f_{tK}(u)g(\nabla h_{tK}(u))dt.$$
Note that $\widehat{h_L}(u) v(u)$ is positive for all $u\in\Omega$. Therefore, by Parseval's type formula (\ref{parseval}),
$$V_{\mu,1}(K,L)=V_{\mu,1}(K,L)=-(2\pi)^{-n}\frac{\pi}{n}\int_{\sfe} \widehat{h_L}(u) P_{\mu,K}(u) du=$$
$$-(2\pi)^{-n}\frac{\pi}{n}\int_{\sfe} \widehat{h_L}(u) P_{\mu,L}(u) du-(2\pi)^{-n}q\epsilon \int_{\Omega} \widehat{h_L}(u) v(u) du<$$
$$-(2\pi)^{-n}\frac{\pi}{n}\int_{\sfe} \widehat{h_L}(u) P_{\mu,L}(u) du=\mu(L).$$
Using the above along with Corollary \ref{p-concave-homo} we get that 
$$\mu(L)> V_{\mu,1}(K,L)\geq \mu(K)^{1-q}\mu(L)^q,$$
and hence $\mu(L)> \mu(K)$. On the other hand, (\ref{p_comp}) implies that $P_{\mu,L}(\theta)\leq P_{\mu,K}(\theta)$ for every $\theta\in\sfe$. 
\end{proof}

\section{Stability and separation for Shephard's problem extension.}

\subsection{Separation result for Theorem \ref{main}.}

\begin{theorem}\label{stability}
Fix $n\geq 1$, $p\in [0,\infty)$ and consider a measure $\mu$ on $\R^n$ whose density $g:\R^n\rightarrow \R^+$ is $p-$concave and $\frac{1}{p}$-homogenous function. Set $q=\frac{p}{np+1}$.

Let $K$ and $L$ be symmetric convex bodies, and let $L$ additionally be a projection body. Fix $\epsilon>0.$ Assume that for every $\theta\in \sfe$ we have 
$$P_{\mu, K}(\theta)\leq P_{\mu, L}(\theta)-\epsilon.$$
Then 
$$\mu(K)^{1-q}\leq \mu(L)^{1-q}-C(\mu)\epsilon,$$ 
where $C(\mu)$ is a constant which only depends on the measure $\mu.$
\end{theorem}

We formulate the following notable corollary of Theorem \ref{stability}.

\begin{corollary}\label{hyperplane}
Fix $n\geq 1$, $p\in [0,\infty)$ and consider a measure $\mu$ on $\R^n$ whose density $g:\R^n\rightarrow \R^+$ is $p-$concave and $\frac{1}{p}$-homogenous function. Set $q=\frac{p}{np+1}$. 

Let $L$ be a strictly convex symmetric projection body. Then 
$$\mu(L)^{1-q}\geq C(\mu)\min_{\theta\in\sfe}P_{\mu,L}(\theta),$$ 
where $C(\mu)$ is a constant which only depends on the measure $\mu.$
\end{corollary}
Corollary \ref{hyperplane} is an analogue of a hyperplane inequality for Lebesgue measure of projections (see Gadrner \cite{gardner}, or Koldobsky \cite{K7}).

\begin{proof}[Proof of Theorem \ref{stability}.]

Let $\tilde{\mu}$ be, as before, the symmetrization of $\mu$, i.e. the measure with the density $g(x)=\frac{g(x)+g(-x)}{2}$.

Assume without loss of generality that $K$ and $L$ are infinitely smooth. The assumptions $\widehat{h_L}\leq 0$ and
$$P_{\mu, K}(\theta)\leq P_{\mu, L}(\theta)-\epsilon,$$
lead to the following chain of inequalities:
$$V_{\tilde{\mu},1}(K,L)=-(2\pi)^{-n}\frac{\pi}{n}\int_{\sfe} \widehat{h_L}(u)P_{\tilde{\mu}, K}(u)du\leq$$
$$-(2\pi)^{-n}\frac{\pi}{n}\int_{\sfe} \widehat{h_L}(u)P_{\tilde{\mu}, L}(u)du+\epsilon (2\pi)^{-n}\frac{\pi}{n}\int_{\sfe}\widehat{h_L}(u)du=$$
$$\tilde{\mu}(L)+\epsilon (2\pi)^{-n}\frac{\pi}{n} \int_{\sfe}\widehat{h_L}(u)du.$$
By Corollary \ref{p-concave-homo}, we have
\begin{equation}\label{food}
\tilde{\mu}(L)+\epsilon (2\pi)^{-n}\frac{\pi}{n} \int_{\sfe}\widehat{h_L}(u)du\geq \tilde{\mu}(K)^{1-q}\tilde{\mu}(L)^q.
\end{equation}
Let $S=\sfe\cap supp(g)$. By Theorem \ref{mink-existence} there exists a symmetric convex body $Q$ (depending on the measure $\mu$) with
$$d\sigma_{\tilde{\mu}, Q}=\frac{1}{q}\widehat{1_S},$$
and therefore satisfying
$$P_{\tilde{\mu},Q}(\theta)=1_S(\theta).$$
We then estimate
$$(2\pi)^{-n}\frac{\pi}{n}\int_{\sfe}\widehat{h_L}(u)du\leq (2\pi)^{-n}\frac{\pi}{n}\int_{S}\widehat{h_L}(u)du=(2\pi)^{-n}\frac{\pi}{n}\int_{\sfe}\widehat{h_L}(u)P_{\tilde{\mu},Q}(\theta)du=$$
\begin{equation}\label{food1}
-V_{\tilde{\mu},1}(Q, L)\leq -\tilde{\mu}(Q)^{1-q}\tilde{\mu}(L)^q.
\end{equation}
Letting $C(\mu)=\tilde{\mu}(Q)^{1-q}$, by (\ref{food}) and (\ref{food1}), we get 
$$\tilde{\mu}(L)-\epsilon C(\mu)\tilde{\mu}(L)^q\geq \tilde{\mu}(K)^{1-q}\tilde{\mu}(L)^q,$$
which implies the statement of the Theorem for $\tilde{\mu}$ in place of $\mu$, and hence the Theorem follows for $\mu$ as well. 
\end{proof}

\subsection{Stability for Theorem \ref{main}.} 

Finally, we prove the stability result.

\begin{theorem}\label{separation}
Fix $n\geq 1$, $p\in [0,\infty)$ and consider a measure $\mu$ on $\R^n$ whose density $g:\R^n\rightarrow \R^+$ is $p-$concave and $\frac{1}{p}$-homogenous function. Set $q=\frac{p}{np+1}$.

Let $K$ and $L$ be symmetric convex bodies, and let $L$ additionally be a projection body. Let $\epsilon>0.$ Assume that for every $\theta\in \sfe$ we have 
$$P_{\mu, K}(\theta)\leq P_{\mu, L}(\theta)+\epsilon.$$
Then $\mu(K)^{1-q}\leq \mu(L)^{1-q}+C(\mu,L)\epsilon$, where $C(\mu,L)$ is a constant which depends on the measure $\mu$ and the body $L.$
\end{theorem}
\begin{proof} Suppose that
$$P_{\mu,K}(\theta)\leq P_{\mu,L}(\theta)+\epsilon.$$
Assume without loss of generality that $K$ and $L$ are infinitely smooth. Then, similarly to the proof of Theorem \ref{stability}, we have
$$\mu(L)-\epsilon (2\pi)^{-n}\frac{\pi}{n} \int_{\sfe}\widehat{h_L}(u)du\geq \mu(K)^{1-q}\mu(L)^q.$$ 
For the unit ball $B_2^n$ we have 
$$(2\pi)^{-n}\frac{\pi}{n}\int_{\sfe}\widehat{h_L}(u)du=-\nu_{n-1}^{-1} V_1(B_2^n,L).$$
Let $R(L)$ be the smallest positive number such that $L\subset R(L) B_2^n$. Note that 
$$V_1(B_2^n,L)=\lim_{\epsilon\rightarrow 0} \frac{|B_2^n+\epsilon L|-|B_2^n|}{n\epsilon}\leq \nu_n \frac{(1+\epsilon R(L))^n-1}{n\epsilon}=\nu_n R(L).$$
Letting $C(L,\mu)=\frac{\nu_n}{\nu_{n-1}} R(L) \mu(L)^{-q}$, we get the statement of the Theorem. 
\end{proof}

\pagebreak
%\section*{Appendix}
\appendix{Appendix}

\begin{lemma}\label{jacapp}
Given convex bodies $K$ and $L$ containing the origin, and a measure $\mu$ with continuous density $g$ on $\R^n$, we have 
$$\mu_1(K,L)=\int_{\sfe} h_L(u) d\sigma_{\mu,K}(u).$$
Here $h_K$ and $h_L$ are support functions of $K$ and $L$ and $\sigma_{\mu,K}$ is the surface area measure of $K$.
\end{lemma}
\begin{proof} Consider a convex compact set $K$. Recall that a unit normal $n_y$ is well defined, continuous and differentiable $H_{n-1}$-almost everywhere for $y\in\partial K$; we shall denote the set where it happens by $\widetilde{\partial K}$. Let $X:\widetilde{\partial K}\times [0,\infty)\rightarrow \R^n\setminus K$ be the map $X(y,t)=y+t n_y$. Let $D(y,t)$ be the Jacobian of this map. Then
$$\frac{1}{\epsilon}\left(\mu(K+\epsilon L)-\mu(K)\right)=\frac{1}{\epsilon}\int_{\widetilde{\partial K}} \int_0^{\epsilon h_L(n_y)} D(y,t) g(y+tn_y) dtdH_{n-1}(y).$$

First, we show that $X(y,t)$ is an expanding map. Let $y_1, y_2\in \partial K$ and $t_1, t_2\in [0,\infty)$. Then
$$|X(y_1,t_1)-X(y_2,t_2)|^2=|y_1+t_1 n_1-y_2-t_2 n_2|^2=$$
\begin{equation}\label{jac1}
|y_1-y_2|^2+|t_1 n_1-t_2 n_2|^2+t_1\langle y_1-y_2, n_1\rangle+t_2\langle y_2-y_1, n_2\rangle.
\end{equation}
By convexity,
$$\langle y_1, n_1\rangle\geq \langle y_2, n_1\rangle,$$
$$\langle y_2, n_2\rangle\geq \langle y_1, n_2\rangle.$$
Hence (\ref{jac1}) is greater than or equal to
$$|y_1-y_2|^2+|t_1 n_1-t_2 n_2|^2\geq |y_1-y_2|^2+|t_1-t_2|^2.$$
This implies that $X(y,t)$ is expanding, and hence $D(y,t)\geq 1.$ Therefore,
$$\mu_1(K,L)\geq \liminf_{\epsilon\rightarrow 0}\frac{1}{\epsilon}\int_{\widetilde{\partial K}} \int_0^{\epsilon h_L(n_y)} g(y+tn_y) dtdH_{n-1}(y)=$$
\begin{equation}\label{fnl}
\int_{\widetilde{\partial K}} h_L(n_y) g(y) d H_{n-1}(y).
\end{equation}
Using the fact that $H_{n-1}(\partial K\setminus \widetilde{\partial K})=0$, and applying the Gauss map to pass the integration on the sphere, we get
$$\mu_1(K,L)\geq \int_{\sfe} h_L(u) d\sigma_{\mu,K}(u).$$

Next, for an arbitrary $\delta>0$, consider a set 
$$(\partial K)_{\delta}=\{y\in\partial K:\, \exists a\in \R^n \,s.t.\, y\in B(a,\delta)\subset K\},$$
where $B(a,\delta)$ stands for a ball of radius $\delta$ centered at $a$. It was shown by  Hug \cite{Hug} (see Besau, Werner \cite{BW} for more details), that the Gauss map is Lipschitz for $y\in (\partial K)_{\delta}.$

For a (small) $\epsilon>0$, assume that $0\leq t_1, t_2\leq \epsilon$, and $y_1, y_2\in (\partial K)_{\delta}.$ Then (\ref{jac1}) is smaller than or equal to
$$|y_1-y_2|^2+|t_1-t_2|^2+\epsilon^2|n_1-n_2|^2 +\epsilon\langle y_1-y_2, n_1-n_2\rangle.$$
Denote by $L(\delta)$ the Lipschitz constant of the Gauss map on $(\partial K)_{\delta}$. Then
$$\frac{|y_1-y_2|^2+|t_1-t_2|^2+\epsilon^2|n_1-n_2|^2 +\epsilon\langle y_1-y_2, n_1-n_2\rangle}{|y_1-y_2|^2+|t_1-t_2|^2}\leq 1+L(\delta)\epsilon+L(\delta)^2\epsilon^2.$$
Therefore,
$$D(y,t)\leq (1+L(\delta)\epsilon+L(\delta)^2\epsilon^2)^{n-1}\leq 1+C(K,n,\delta)\epsilon.$$
Hence, in view of (\ref{fnl}), the limit in $\epsilon$ exists, and
$$\lim_{\epsilon\rightarrow 0}\frac{1}{\epsilon}\int_{(\partial K)_{\delta}} \int_0^{\epsilon h_L(n_y)} D(y,t)g(y+tn_y) dtdH_{n-1}(y)=\int_{(\partial K)_{\delta}} h_L(n_y) g(y) d H_{n-1}(y),$$
and by dominated convergence theorem and lower-semi continuity,
$$\mu_1(K,L)=\liminf_{\epsilon\rightarrow 0}\frac{1}{\epsilon}\int_{\partial K} \int_0^{\epsilon h_L(n_y)} D(y,t)g(y+tn_y) dtdH_{n-1}(y)=$$
$$\lim_{\epsilon\rightarrow 0}\lim_{\delta\rightarrow 0}\frac{1}{\epsilon}\int_{(\partial K)_{\delta}} \int_0^{\epsilon h_L(n_y)} D(y,t)g(y+tn_y) dtdH_{n-1}(y)=$$
$$\lim_{\delta\rightarrow 0}\lim_{\epsilon\rightarrow 0}\frac{1}{\epsilon}\int_{(\partial K)_{\delta}} \int_0^{\epsilon h_L(n_y)} D(y,t)g(y+tn_y) dtdH_{n-1}(y)=$$
$$\lim_{\delta\rightarrow 0}\int_{(\partial K)_{\delta}} h_L(n_y) g(y) d H_{n-1}(y)=$$
$$\int_{\widetilde{\partial K}} h_L(n_y) g(y) d H_{n-1}(y)=\int_{\partial K} h_L(n_y) g(y) d H_{n-1}(y)=$$
$$\int_{\sfe} h_L(u) d\sigma_{\mu,K}(u).$$
The last equation is obtained via the application of the Gauss map.
\end{proof}

\begin{proposition}\label{p-concave-proposition}
For $p\geq 0$ and $r\geq 0$, let $g:\R^n\rightarrow \R^+$ be $p-$concave and $r$-homogenous. Then $g$ is also $\frac{1}{r}-$concave.
\end{proposition}
\begin{proof} The proof splits in two cases. Firstly, if $\frac{1}{r}\leq p$, then the statement follows automatically by the standard inequality for $q-$averages
$$M_q(\lambda, a,b)\leq M_{q'}(\lambda, a,b),$$
whenever $q\leq q'$ (see the definition (\ref{p-av}) and Gardner \cite{gardner} for more details).

Secondly, let $0\leq r\leq \frac{1}{p}$. Observe, that in the presence of $r-$homogeneity it is sufficient to show that for every $x,y\in\R^n$ one has
\begin{equation}\label{goalconc}
g(x+y)\geq \left(g(x)^{\frac{1}{r}}+g(y)^{\frac{1}{r}}\right)^r.
\end{equation}
By $p-$concavity, we have for every $\lambda\in [0,1]$:
$$g(x+y)=g\left(\lambda \frac{x}{\lambda}+(1-\lambda)\frac{y}{1-\lambda}\right)\geq \left(\lambda g\left(\frac{x}{\lambda}\right)^p+(1-\lambda)g\left(\frac{y}{1-\lambda}\right)^p\right)^{\frac{1}{p}}=$$
\begin{equation}\label{lambda}
\left(\lambda^{1-pr} g(x)^p+(1-\lambda)^{1-pr}g(y)^p\right)^{\frac{1}{p}}.
\end{equation}

Observe that for 
$$\lambda_0=\frac{g(x)^{\frac{1}{r}}}{g(x)^{\frac{1}{r}}+g(y)^{\frac{1}{r}}},$$
the expression in (\ref{lambda}) is exactly equal to the right hand side of (\ref{goalconc}), which concludes the proof. 
\end{proof}

We remark that $\lambda_0$ in the proof above is found as the maximizer for the function from (\ref{lambda}).

\begin{proposition}\label{app}
Let $K$ and $L$ be convex bodies within Hausdorff distance $\epsilon$ from each other, $\epsilon>0$. Let $\mu$ be a measure on $\R^n$ with density $g(x)$, continuous on its support. Then for every Lipschitz function $a(u),$
$$\left|\int_{\sfe} a(u) d\sigma_{\mu,K}(u)-\int_{\sfe} a(u) d\sigma_{\mu,L}(u)\right|\leq C(\epsilon),$$
where the constant $C(\epsilon)>0$ depends on $a(u),$ $g(x)$, $K$ and $L$, and tends to zero when $\epsilon\rightarrow 0.$
\end{proposition}
\begin{proof} We write
$$\left|\int_{\sfe} a(u) d\sigma_{\mu,K}(u)-\int_{\sfe} a(u) d\sigma_{\mu,L}(u)\right|=$$
$$\left|\int_{\sfe} a(u) g(\nu_K^{-1}(u))d\sigma_{K}(u)-\int_{\sfe} a(u) g(\nu_L^{-1}(u))d\sigma_{L}(u)\right|\leq$$
\begin{equation}\label{equ1}
\int_{\sfe} |a(u)| \left|g(\nu_K^{-1}(u))-g(\nu_L^{-1}(u))\right|d\sigma_{K}(u)+
\end{equation}
\begin{equation}\label{equ2}
\left|\int_{\sfe} a(u) g(\nu_L^{-1}(u))d\sigma_{K}(u)-\int_{\sfe} a(u) g(\nu_L^{-1}(u))d\sigma_{L}(u)\right|.
\end{equation}
Since $K$ and $L$ are convex bodies, and hence are bounded, $g(x)$ is uniformly continuous on their boundary. Hence, as the Hausdorff distance between $K$ and $L$ is bounded by $\epsilon,$
$$|g(\nu_K^{-1}(u))-g(\nu_L^{-1}(u))|\leq C'|\nu_K^{-1}(u)-\nu_L^{-1}(u)|,$$
and thus, by the weak convergence of the inverse Gauss maps of convex bodies converging in Hausdorff distance (see, e.g. Schneider \cite{book4}),
$$\int_{\sfe} |a(u)| \left|g(\nu_K^{-1}(u))-g(\nu_L^{-1}(u))\right|d\sigma_{K}(u)\leq C'(\epsilon),$$
where $C'(\epsilon)\rightarrow 0$ as $\epsilon\rightarrow 0.$
As $a(u)$ is a continuous function on $\sfe,$ it attains its maximum. Hence there exists a constant $C''(\epsilon)$, depending on $a(u)$, $g(x)$, $K$ and $L$ such that (\ref{equ1}) is bounded from above by $C''(\epsilon),$ and $C''(\epsilon)$ tends to zero as $\epsilon\rightarrow 0.$ 

Next, (\ref{equ2}) is bounded from above by 
$$\tilde{C}\left|\int_{\sfe} a(u) d\sigma_{K}(u)-\int_{\sfe} a(u) d\sigma_{L}(u)\right|,$$
which in turn is bounded by $\tilde{C}'(\epsilon)\rightarrow_{\epsilon\rightarrow 0} 0,$ since classical (Lebesgue) surface area measures of convex bodies, which converge in Hausdorff distance, do converge weakly (see, e.g. Schneider \cite{book4}). The proposition follows. 
\end{proof}

\begin{proposition}\label{2ref}
If a body $K$ is $C^2-$smooth and strictly convex then its surface area measure with respect to a measure $\mu$ with density $g$, continuous on its support, has representation
$$d\sigma_{\mu,K}(u)=f_K(u)g(\nabla h_K(u))du.$$
\end{proposition}
\begin{proof} Under the assumptions of the proposition, the Gauss map $\nu_K$ of $K$ is a bijection, and $\nu_K^{-1}(u)=\nabla h_K(u)$ for every $u\in\sfe$. Therefore, for every $\Omega\subset \sfe$,
$$\sigma_{\mu,K}(\Omega)=\int_{\nu_K^{-1}(\Omega)} g(x) d\sigma_K(x)=$$
$$\int_{\Omega} g(\nu_K^{-1}(u)) f_K(u) du=\int_{\Omega} f_K(u)g(\nabla h_K(u))du.$$
\end{proof}

\begin{proposition}\label{1ref}
The surface area measure of a convex polytope $P$ with respect to a measure $\mu$ has representation
$$d\sigma_{\mu,P}(u)=\sum_{i=1}^N \delta_{u_i} \mu_{n-1}(F_i),$$
where $u_i$, $i=1,...,N$ are the normals to the faces of the polytope, $F_i$ are the corresponding faces, and $\mu_{n-1}(F_i)$ stands for $\int_{F_i} d\mu(x)$.
\end{proposition}
\begin{proof} For a polytope $P$ with faces $F_i$ and corresponding normals $u_i$, Gauss map $\nu_K$ is defined everywhere in the interior of the faces, and for $x\in int(F_i)$, $\nu_K(x)=u_i$. Hence, for a Borel set $\Omega\subset \sfe,$
$$\sigma_{\mu,K}(\Omega)=\int_{\nu_K^{-1}(\Omega)} g(x) d\sigma_K(x)=\sum_{i:\,u_i\in\Omega} \int_{F_i} d\mu_{n-1}(x).$$
\end{proof}

\end{document}